\documentclass[11pt]{amsart}
\usepackage{amsfonts}
\usepackage{hyperref}
\usepackage[utf8]{inputenc}
\usepackage[T1]{fontenc}
\usepackage[dvips]{graphicx}
\usepackage{amssymb}
\usepackage{amsmath}
\usepackage{dsfont}
\usepackage{color}
\usepackage{latexsym}
\usepackage{bbm}
\usepackage{color}
\usepackage{amsthm}
\usepackage{multicol}
\usepackage[top   = 2.75cm,
bottom = 2.50cm,
left   = 2.50cm,
right  = 2.00cm]{geometry}

\bibstyle{plain}
\theoremstyle{plain}
\newtheorem{theorem}{Theorem}
\newtheorem{lemma}[theorem]{Lemma}

\newtheorem{definition}[theorem]{Definition}

\theoremstyle{remark}

\newtheorem{remark}[theorem]{Remark}

\newcommand{\N}{\mathbb N}

\begin{document}

\title[]{A Continuous digit projector from binary representations of numbers onto $A_2$-representation
}
\author[O.\ O.\ Nikorak, S.\ P.\ Ratushniak]{ O.\ O.\ Nikorak, S.\ P.\ Ratushniak}

\newcommand{\eacr}{\newline\indent}

\address{O.O. Nikorak\eacr
Institute of Mathematics of NASU, Kyiv,
Ukraine\acr ORCID 0009-0009-2269-7826}
\email{nikorak@imath.kiev.ua}

\address{S.P. Ratushniak\eacr
Institute of Mathematics of NASU,
 Dragomanov Ukrainian State University, Kyiv,
Ukraine\acr ORCID 0009-0005-2849-6233}
\email{ratush404@gmail.com}


\date{\today}

\newcommand{\acr}{\newline\indent}

\begin{abstract}
As is known, the $A_2$-continued representation of numbers is not topologically equivalent to the classical binary representation; therefore, the digit projector of such representations is a discontinuous function. In this paper, we introduce a continuous function that serves as an analogue of the digit projector of the classical binary representation of numbers into the digits of the  $A_2$-continued representation with zero redundancy, namely a function of the form
\[f(\Delta^2_{\alpha_1\alpha_2...\alpha_{2n-1}\alpha_{2n}...})=
\Delta^{A_2}_{(\frac{1}{2})^{1-\alpha_1}(\frac{1}{2})^{\alpha_2}...
(\frac{1}{2})^{1-\alpha_{2n-1}}(\frac{1}{2})^{\alpha_{2n}}...}, \alpha_n\in \{0,1\}.\]
It is proved that the function $f$ is well-defined, continuous, and monotone. Using the normal properties of numbers with respect to their binary representation and Lebesgue's theorem asserting the existence of a finite derivative for a continuous monotone function almost everywhere, the singularity of the function 
$f$ is established.

The paper also establishes a relationship between the considered function, the right-shift operator on digits, and the inversor of the continued representation of numbers. This relationship is then used to establish the singularity of the inversor.

\end{abstract}

\maketitle

\section*{Introduction}
Functions with complicated local properties often arise as mappings generated by digit transformations from one representation of numbers to another. These representations may coincide or may be consistent with respect to their geometry or alphabets. Of particular interest is the class of continuous functions for which the corresponding systems of representations have alphabets of the same cardinality, but are not necessarily consistent with respect to the geometry of the representations.

Among interesting examples of such functions with complicated differential properties defined in terms of continued fraction representations of numbers are the Minkowski singular function~\cite{salem}, the Sendov-type singular function~\cite{PGLR}, continuous nowhere monotonic functions of the Tribin-function type~\cite{rat,ch}, and others. The investigation of their differential properties requires individual methods in each particular case. At present, we are not aware of any general methods for proving the singularity of functions defined by continued fraction representations of at least one of their variables.

In the present paper, we study a continuous function whose values are represented by infinite $A_2$-continued fractions. Using the ideas developed in~\cite{PGLR}, we prove that this function is singular.
            
\section{The $A_2$-continued fraction representation of real numbers}

A continued fraction is an expression of the form
\begin{eqnarray*}
a_0+\frac{1}{a_1+\frac{1}{a_2+\frac{1}{a_3+_{\ddots}}}}\equiv
a_0+1/a_1+1/a_2+1/a_3...\equiv[a_0;a_1,a_2,a_3,\ldots],
\end{eqnarray*}
where $a_i\in \mathbb{R}$, $i\in \mathbb{N}$, and $a_0\in \mathbb{N}_0$. The convergence of an infinite continued fraction is guaranteed by Seidel's theorem~\cite{seidel}.

It is well known that every real number $x\in \mathbb{R}$ can be represented by a regular continued fraction, that is, a continued fraction whose partial quotients satisfy $a_n\in \mathbb{N}$.
The value of a finite continued fraction $[a_0;a_1,a_2,\ldots,a_n]$ is a rational number, i.e.,
$$
[a_0;a_1,a_2,\ldots,a_n]=\frac{p_n}{q_n},
$$
where the fraction $\frac{p_n}{q_n}$ is irreducible. The fraction $\frac{p_n}{q_n}$ is called a \emph{convergent} of the continued fraction $[a_0;a_1,a_2,\ldots,a_n]$.

The following recurrence relations hold for the convergents:
\[
\begin{cases}
   p_n=a_np_{n-1}+p_{n-2}, \\
   q_n=a_nq_{n-1}+q_{n-2},
\end{cases}
\quad n\in \mathbb{N},\qquad \mbox{where }
\begin{cases}
  p_0=a_0, \\
  q_0=1,
\end{cases}
\;
\begin{cases}
  p_{-1}=1, \\
  q_{-1}=0.
\end{cases}
\]

The following properties of convergents hold~\cite{pglr_umz}: for every $k\in N$

\begin{enumerate}
  \item[1)] $q_kp_{k-1}-p_kq_{k-1}=(-1)^k$;
  \item[2)] $\dfrac{q_k}{q_{k-1}}=[a_k;a_{k-1},\ldots,a_1]$;
  \item[3)] $[a_0;a_1,a_2,\ldots,a_k]=\dfrac{p_{k-1}r_k+p_{k-2}}{q_{k-1}r_k+q_{k-2}}$, where $r_k=[a_k;a_{k+1},a_{k+2},\ldots]$.
\end{enumerate}

For the infinite continued fraction $[a_0;a_1,\ldots,a_n,\ldots]$ and the corresponding convergents $\frac{p_n}{q_n}=[a_0;a_1,\ldots,a_n]$, the following equality holds:
\[
\lim\limits_{n\to\infty}\frac{p_n}{q_n}=[a_0;a_1,\ldots,a_n,\ldots].
\]

Let us consider an encoding of real numbers from the interval $[\frac{1}{2},1]$ by infinite continued fractions, namely, by means of $A_2$-continued fractions.

Let $A_2\equiv\{\frac{1}{2},1\}$ be an alphabet and let $L_2\equiv A_2\times A_2\times\ldots$ be the space of sequences over the alphabet $A_2$.

An \emph{$A_2$-continued fraction} (or simply an \emph{$A_2$-fraction}) is a continued fraction of the form
$$
0+1/a_1+1/a_2+\ldots+1/a_n+\ldots,\qquad \mbox{where } a_n\in A_2.
$$

\begin{theorem}\cite{r2}
For every number $x\in[\frac{1}{2},1]$, there exists a sequence $(a_n)\in L_2$ such that
\begin{equation}\label{rA2}
x=1/a_1+1/a_2+\ldots+1/a_n+\ldots\equiv
\Delta^{A_2}_{a_1a_2\ldots a_n\ldots}.
\end{equation}
\end{theorem}

The expansion of a number $x$ into the continued fraction $[0;a_1,\ldots,a_n,\ldots]$ is called its \emph{$A_2$-continued fraction representation}, while the notation $\Delta^{A_2}_{a_1a_2\ldots a_n\ldots}$ is called its \emph{$A_2$-representation}.

\begin{definition}\cite{r2}
A cylinder of rank $m$ with base $c_1c_2\ldots c_m$ is the set $\Delta_{c_1c_2\ldots c_m}^{A_2}$ consisting of all numbers of the form $\Delta^{A_2}_{c_1c_2\ldots c_ma_1a_2\ldots}$, i.e.,
$$
\Delta_{c_1c_2\ldots c_m}^{A_2}
=
\{x:\,
x=\Delta_{c_1c_2\ldots c_ma_1a_2\ldots a_n\ldots}^{A_2},
\;
a_n\in A_2,\ \forall~ n\in \mathbb{N}
\}.
$$
\end{definition}

The cylinder $\Delta_{c_1c_2\ldots c_m}^{A_2}$ is an interval with endpoints
\begin{eqnarray}
\nonumber
[0;c_1,\ldots,c_m,(e_1,e_0)]
\mbox{ and }
[0;c_1,\ldots,c_m,(e_0,e_1)],
\end{eqnarray}
where the left endpoint and the right endpoint depend on the parity of $m$.

The diameter (length) of the cylinder $\Delta_{c_1c_2\ldots c_m}^{A_2}$ is given by the formula~\cite{r2,n7}
\begin{eqnarray}
|{\Delta_{c_1c_2\ldots c_m}^{A_2}}|
=
\frac{1}{(q_{m-1}+q_m)(q_{m-1}+2q_m)}
=
\frac{2c^2+1+2c\frac{q_{m-1}}{q_m}}
{1+c\frac{q_{m-1}}{q_m}}
|{\Delta_{c_1c_2\ldots c_mc}^{A_2}}|.
\end{eqnarray}

The topological, metric, and probabilistic theories of the $A_2$-continued fraction representation of numbers are developed in~\cite{n2,PratsKyurchev2009,pm_2025}.

\section{The Main Object of Study}
 To define the main object of study, let us recall that by the classical binary representation of numbers we mean a representation of the form $\Delta^2_{\alpha_1\alpha_2...\alpha_n...}$, where $\alpha_n\in\{0,1\}\equiv A$, and its value is given by
\[
\Delta^2_{\alpha_1\alpha_2...\alpha_n...}=
\sum\limits_{n=1}^{\infty}\frac{\alpha_n}{2^n}=
\bigcap\limits_{n=1}^{\infty}\Delta^2_{\alpha_1...\alpha_n}, \qquad
\Delta^2_{\alpha_1...\alpha_n}=
\left[
\sum\limits_{i=1}^{n}\frac{\alpha_i}{2^i};
\sum\limits_{i=1}^{n}\frac{\alpha_i}{2^i}+\frac{1}{2^n}
\right].
\]

Consider the function $f$ defined by
\begin{equation}\label{eq:f}
f(x)=f(\Delta_{\alpha_1\alpha_2...\alpha_{2n-1}\alpha_{2n}...}^2)=
\Delta^{A_2}_{\left(\frac{1}{2}\right)^{[1-\alpha_1]}
\left(\frac{1}{2}\right)^{\alpha_2}
...
\left(\frac{1}{2}\right)^{[1-\alpha_{2n-1}]}
\left(\frac{1}{2}\right)^{\alpha_{2n}}
...}.
\end{equation}
The function $f$ is well-defined despite the fact that binary-rational numbers $\Delta^2_{\alpha_1...\alpha_{n-1}0(1)}=\Delta^2_{\alpha_1...\alpha_{n-1}1(0)}$ possess two formally distinct representations. Indeed,
\[
f(\Delta^2_{\alpha_1...\alpha_{2n}(0)})=
\Delta^{A_2}_{\left(\frac{1}{2}\right)^{[1-\alpha_1]}...
\left(\frac{1}{2}\right)^{\alpha_{2n}}
(\frac{1}{2}1)}
=
\Delta^{A_2}_{\left(\frac{1}{2}\right)^{[1-\alpha_1]}...
\left(\frac{1}{2}\right)^{\alpha_{2n}-1}
(1\frac{1}{2})}
=
f(\Delta^2_{\alpha_1...[\alpha_{2n}-1](1)}),
\]

\[
f(\Delta^2_{\alpha_1...\alpha_{2n-1}(0)})=
\Delta^{A_2}_{\left(\frac{1}{2}\right)^{[1-\alpha_1]}...
\left(\frac{1}{2}\right)^{[1-\alpha_{2n-1}]}
(1\frac{1}{2})}
=
\Delta^{A_2}_{\left(\frac{1}{2}\right)^{[1-\alpha_1]}...
\left(\frac{1}{2}\right)^{2-\alpha_{2n-1}}
(\frac{1}{2}1)}
=
f(\Delta^2_{\alpha_1...[\alpha_{2n-1}-1](1)}).
\]

\begin{theorem}
The function $f(x)$ is continuous and strictly decreasing. Moreover,
$f(0)=1$ and $f(1)=\frac{1}{2}$.
\end{theorem}

\begin{proof}
We prove the continuity of $f$ separately on the sets of binary-rational and binary-irrational numbers

Let
$x_0=\Delta^{2}_{\alpha_1...\alpha_{n}...}$
be a binary-irrational point of the interval $[0,1]$. The function $f$ is continuous at $x_0$ if, for every sequence $(x_k)$ of binary-irrational numbers converging to $x_0$ 
\[
\lim\limits_{x\to x_k}|f(x)-f(x_k)|=0.
\]

For $x\neq x_0$, let $m$ be the first index such that
$\alpha_m(x)\neq\alpha_m(x_0)$. Then
$\alpha_i(x)=\alpha_i(x_0)$ for all $i<m$.
Thus, $x\to x_0$ if and only if $m\to\infty$.
Without loss of generality, assume that
$\alpha_m(x_0)=\alpha_m$. Then
\[
x=\Delta^2_{\alpha_1...\alpha_{m-1}[1-\alpha_m]\alpha'_{m+1}\alpha'_{m+2}...}.
\]

Let
\begin{eqnarray*}
f(x_0)&=&
\Delta^{A_2}_{b_1b_2...b_{m-1}b_m...}
=
\frac{r_{m+1}p_m+p_{m-1}}
{r_{m+1}q_m+q_{m-1}},
\\
&&
r_{m+1}=[b_{m+1};b_{m+2},b_{m+3}...],
\end{eqnarray*}
and
\begin{eqnarray*}
f(x)&=&
\Delta^{A_2}_{b_1b_2...b_{m-1}b'_m...}
=
\frac{r'_{m+1}p_m+p_{m-1}}
{r'_{m+1}q_m+q_{m-1}},
\\
&&
r'_{m+1}=[b'_{m+1};b'_{m+2},b'_{m+3}...].
\end{eqnarray*}

Then
\begin{align*}
\lim\limits_{x\to x_0}|f(x)-f(x_0)|
&=
\lim\limits_{m\to\infty}
\left|
\frac{r_{m+1}p_m+p_{m-1}}
{r_{m+1}q_m+q_{m-1}}
-
\frac{r'_{m+1}p_m+p_{m-1}}
{r'_{m+1}q_m+q_{m-1}}
\right|
\\
&=
\lim\limits_{m\to\infty}
\frac{|r_{m+1}-r'_{m+1}|}
{(r_{m+1}q_m+q_{m-1})(r'_{m+1}q_m+q_{m-1})}
=0.
\end{align*}

Since the point $x_0$ was chosen arbitrarily, it follows that the function $f$ is continuous on the set of binary-irrational numbers.

The continuity of the function $f$ on the set of binary-rational numbers follows immediately from the correctness of its definition at binary-rational points.
Consider two arbitrary numbers
$x_1=\Delta_{\alpha_1\alpha_2...\alpha_n...}^2$
and
$x_2=\Delta_{\beta_1\beta_2...\beta_n...}^2$
such that $x_1<x_2$. Then there exists $n\in \mathbb{N}$ such that one of the following two cases occurs:
\begin{align*}
\alpha_i=\beta_i,\quad i=1,2,\ldots,2n,
\mbox{ but }
0=\alpha_{2n+1}<\beta_{2n+1}=1,\\
\alpha_i=\beta_i,\quad i=1,2,\ldots,2n-1,
\mbox{ but }
1=\alpha_{2n}>\beta_{2n}=0.
\end{align*}

Using the comparison rules for numbers represented by $A_2$-continued fraction representations~\cite{r2}, the corresponding values of $f$ satisfy 
\begin{align*}
f(x_1)&=
\Delta^{A_2}_{(\frac{1}{2})^{[1-\alpha_1]}
(\frac{1}{2})^{\alpha_2}
...
(\frac{1}{2})^{\alpha_{2n}}
\frac{1}{2}...}
>
\Delta^{A_2}_{(\frac{1}{2})^{[1-\alpha_1]}
(\frac{1}{2})^{\alpha_2}
...
(\frac{1}{2})^{\alpha_{2n}}
1...}
=f(x_2),\\
f(x_1)&=
\Delta^{A_2}_{(\frac{1}{2})^{[1-\alpha_1]}
(\frac{1}{2})^{\alpha_2}
...
(\frac{1}{2})^{\alpha_{2n-1}}
1...}
>
\Delta^{A_2}_{(\frac{1}{2})^{[1-\alpha_1]}
(\frac{1}{2})^{\alpha_2}
...
(\frac{1}{2})^{\alpha_{2n-1}}
\frac{1}{2}...}
=f(x_2).
\end{align*}

Hence, the inequality $x_1<x_2$ implies that $f(x_1)>f(x_2)$. Therefore, the function $f$ is strictly decreasing. Moreover,
\[
f(0)=f(\Delta_{(0)}^2)=\Delta_{(\frac{1}{2}1)}^{A_2}=1,
\qquad
f(1)=f(\Delta_{(1)}^2)=\Delta_{(1\frac{1}{2})}^{A_2}=\frac{1}{2}.
\qedhere\]
\end{proof}

\begin{theorem}
The function $f$ is singular.
\end{theorem}

\begin{proof}
Since  $f$ is continuous and monotone, Lebesgue's theorem implies that $f$ has a finite derivative almost everywhere on its domain. Denote the set of all such points by $D$.

Consider a point $x_0\in D$ which is a normal number~\cite{PGLR}, i.e., a number whose binary representation contains every finite block of zeros and ones infinitely often. It is well known that the set $H$ of all such numbers has full Lebesgue measure. Hence,
$x_0\in W=D\cap H$. Therefore, $$\lambda(W)=1,$$ 
since $W$ is the intersection of two sets of full Lebesgue measure.
Let
\[
f(x_0)=f(\Delta^{2}_{\alpha_1...\alpha_n...})=
\Delta^{A_2}_{a_1a_2...a_n...}.
\]
By the definition of the cylinder derivative~\cite{n7}, we obtain
\begin{align*}
-f'(x_0)
&=
\lim\limits_{n\to\infty}
\frac{|f(\Delta^{2}_{\alpha_1...\alpha_n})|}
{|\Delta^{2}_{\alpha_1...\alpha_n}|}
=
\lim\limits_{n\to\infty}
\frac{|\Delta^{A_2}_{a_1...a_n}|}
{|\Delta^{2}_{\alpha_1...\alpha_n}|}=
\\
&=
\lim\limits_{n\to\infty}
\frac{|\Delta^{A_2}_{a_1...a_n}|}
{\frac{1}{2^n}}
=
\frac{1}{2}
\lim\limits_{n\to\infty}
\prod\limits_{k=1}^{n}
\frac{2|\Delta^{A_2}_{a_1...a_k}|}
{|\Delta^{A_2}_{a_1...a_{k-1}}|}=
\\
&=
\frac{1}{2}
\prod\limits_{n=1}^{\infty}
\frac{2|\Delta^{A_2}_{a_1...a_n}|}
{|\Delta^{A_2}_{a_1...a_{n-1}}|}
=
\frac{1}{2}
\prod\limits_{n=1}^{\infty}
\frac{2\left(1+a\frac{q_{n-1}}{q_n}\right)}
{2a^2+1+2a\frac{q_{n-1}}{q_n}}=
\\
&=
\frac{1}{2}
\prod\limits_{n=1}^{\infty}
\frac{\frac{2}{a}+2\frac{q_{n-1}}{q_n}}
{3+2\frac{q_{n-1}}{q_n}},
\qquad
a=a_n\in\left\{\frac{1}{2},1\right\}.
\end{align*}

Since the $n$-th factor of the infinite product does not tend to $1$ as $n\to\infty$, the necessary condition for the convergence of an infinite product is not satisfied. Therefore, for every $x_0\in W$,
\[
f'(x_0)=0
\]
almost everywhere with respect to the Lebesgue measure. Hence, the function $f$ is singular.
\end{proof}

\begin{lemma}
The graph $\Gamma_f$ of the function $f$ is invariant under the following family of transformations
\[
\Gamma_f=\bigcup\limits_{(i,j)\in A\times A}g_{ij}(\Gamma_f),
\]
where
\[
g_{ij}:
\begin{cases}
x'=\dfrac{x}{2^2}+\dfrac{2i+j}{2^2},\\[2mm]
y'=\dfrac{1}{2^{\,i-1}+\dfrac{1}{2^{-j}+y}},
\end{cases}
\qquad
(i,j)\in A\times A,
\]
and
\[
g_{ij}(M(x,y))=M'(x',y').
\]
Hence, the graph of $f$ is invariant under the family $\{g_{ij}\}$.
\end{lemma}

\begin{proof}
Consider an arbitrary point $M(x',y')\in\Gamma_f$. Then
\[
M(x',y')\in
\Delta^2_{ij}\times
\Delta^{A_2}_{(\frac12)^{[1-i]}(\frac12)^j},
\]
where $(i,j)\in A\times A$.

According to the definition of the function,
\[
x'
=
\Delta^2_{ij\alpha_1\alpha_2...\alpha_{2n-1}\alpha_{2n}...},
\]
and
\[
y'
=
f(x)
=
\Delta^{A_2}_{(\frac12)^{[1-i]}
(\frac12)^j
(\frac12)^{[1-\alpha_1]}
(\frac12)^{\alpha_2}
...
(\frac12)^{[1-\alpha_{2n-1}]}
(\frac12)^{\alpha_{2n}}
...}.
\]

Then, for $x'\in\Delta^2_{ij}$, we obviously have
\[
x'
=
\Delta^2_{ij\alpha_1\alpha_2...\alpha_{2n-1}\alpha_{2n}...}
=
\frac{i}{2}
+\frac{j}{2^2}
+\frac{x}{2^2},
\]
where
\[
x=
\Delta^2_{\alpha_1\alpha_2...\alpha_{2n-1}\alpha_{2n}...}
\in[0,1],
\qquad
\alpha_i\in A,\ i\in\N,
\]
and
\[
y'
=
\Delta^{A_2}_{(\frac12)^{[1-i]}
(\frac12)^j
(\frac12)^{[1-\alpha_1]}
(\frac12)^{\alpha_2}
...
(\frac12)^{[1-\alpha_{2n-1}]}
(\frac12)^{\alpha_{2n}}
...}
=
\frac{1}
{(\frac12)^{[1-i]}+
\frac{1}
{(\frac12)^j+y}},
\]
where
\[
y=
\Delta^{A_2}_{(\frac12)^{[1-\alpha_1]}
(\frac12)^{\alpha_2}
...
(\frac12)^{[1-\alpha_{2n-1}]}
(\frac12)^{\alpha_{2n}}
...}
\in\left[\frac12,1\right],
\qquad
\alpha_i\in A,\ i\in\N.
\]

Conversely, for every $M(x,y)\in \Gamma_f$ and every $(i,j)\in A^2$, the point $g_{ij}(M(x,y))$ belongs to
$\Gamma_f$ by the definition of $f$. Hence, 
\[\Gamma_f=\bigcup\limits_{(i,j)\in A^2}g_{ij}(\Gamma_f).\qedhere\]
\end{proof}

\section{A Further Singular Function}
\begin{theorem}
The function $\bar{f}$ defined by
\begin{equation}\label{eq:f1}
\bar{f}(x)=\bar{f}(\Delta_{\alpha_1\alpha_2...\alpha_{2n}
\alpha_{2n+1}...}^2)=
\Delta^{A_2}_{\left(\frac{1}{2}\right)^{\alpha_1}
\left(\frac{1}{2}\right)^{[1-\alpha_2]}
...
\left(\frac{1}{2}\right)^{\alpha_{2n-1}}
\left(\frac{1}{2}\right)^{[1-\alpha_{2n}]}
...},
\end{equation}
is continuous, strictly increasing, and singular.
\end{theorem}
\begin{proof}
The result follows immediately from the relation between
$\bar{f}$ and $f$:
\[
\bar{f}(x)=\omega(f(x)),
\]
where
\[
\omega(x)=
\omega(\Delta^{A_2}_{a_1a_2...a_n...})
=
\Delta^{A_2}_{a_2...a_n...}
=
\frac{1}{x}-a_1(x)
\]
is the digit shift operator for the $A_2$-continued fraction representation of numbers~\cite{r5}.
\end{proof}

\begin{remark}
The inverse function of $\bar{f}$ is a continuous singular function.
\end{remark}

\begin{remark}
For inversor of the $A_2$-continued fraction representation of numbers, the following equality holds:
\[
I(\Delta^{A_2}_{(\frac{1}{2})^{\alpha_1}
(\frac{1}{2})^{\alpha_2}
...
(\frac{1}{2})^{\alpha_n}
...})
=
\Delta^{A_2}_{(\frac{1}{2})^{[1-\alpha_1]}
(\frac{1}{2})^{[1-\alpha_2]}
...
(\frac{1}{2})^{[1-\alpha_n]}
...}
=
f(\bar{f}^{-1}(x)).
\]
\end{remark}

Thus, the inversor $I$ of the $A_2$-continued fraction representation of numbers is the composition of two singular functions (which are not mutually inverse)~\cite{PratsKosoplyot_2002}. Therefore, the digit inversor is also a singular function.

\end{document}